\documentclass[a4paper,11pt]{amsart}
\usepackage[latin1]{inputenc}
\usepackage[english]{babel}
\usepackage{amsmath, amsthm, amssymb, amsopn, amsfonts, amstext, stmaryrd, enumerate, color, mathtools, hyperref, tensor, mathrsfs}
\numberwithin{equation}{section}

\newcommand{\C}{\mathscr{C}}
\newcommand{\E}{\mathcal{E}}

\renewcommand{\L}{\mathscr{L}}
\newcommand{\N}{\mathbb{N}}
\newcommand{\R}{\mathbb{R}}

\newcommand{\Z}{\mathbb{Z}}

\newcommand{\loc}{{\rm loc}}
\newcommand{\dive}{{\mbox{\normalfont div}}}

\newcommand{\dist}{{\mbox{\normalfont dist}}}
\newcommand{\bequ}{\begin{equation}}
\newcommand{\nequ}{\end{equation}}
\newcommand{\mini}[1]{\mbox{\tiny $#1$}}

\newcommand{\Deltaz}{\tensor*[]{\Delta}{_z}}
\newcommand{\Deltamz}{\tensor*[]{\Delta}{_{-z}}}

\newcommand{\PV}{\mbox{\normalfont P.V.}}
\newcommand{\Haus}{\mathcal{H}}

\DeclareMathOperator{\supp}{supp}

\theoremstyle{plain}
\newtheorem{definition}{Definition}[section]
\newtheorem{theorem}[definition]{Theorem}
\newtheorem{proposition}[definition]{Proposition}
\newtheorem{lemma}[definition]{Lemma}

\theoremstyle{definition}
\newtheorem{remark}[definition]{Remark}

\renewcommand{\le}{\leqslant}

\renewcommand{\ge}{\geqslant}

\title[Interior regularity in Sobolev and Nikol'skii spaces]{Interior regularity of solutions of non-local
equations in Sobolev and Nikol'skii spaces}

\author{Matteo Cozzi}

\address{Dipartimento di Matematica ``Federigo Enriques'', Universit\`a degli Studi di Milano, Via Saldini 50, I-20133 Milano (Italy) and Laboratoire Ami\'enois de Math\'ematique Fondamentale et Appliqu\'ee, UMR CNRS 7352, Universit\'e de Picardie ``Jules Verne'', 33 Rue Saint-Leu, F-80039 Amiens (France)}

\email{matteo.cozzi@unimi.it}



\thanks{It is a pleasure to thank Enrico Valdinoci for his dedication, advice and encouragement. I would also like to express my gratitude to the Weierstra{\ss} Institut f\"ur Angewandte Analysis und Stochastik (WIAS) of Berlin, where part of this work was carried out.}

\begin{document}

\maketitle

\bigskip
\bigskip

\begin{abstract}
We prove interior~$H^{2 s - \varepsilon}$ regularity for weak solutions of linear elliptic integro-differential equations close to the fractional~$s$-Laplacian. The result is obtained via intermediate estimates in Nikol'skii spaces, which are in turn carried out by means of an appropriate modification of the classical translation method by Nirenberg.
\end{abstract}

\bigskip

\section{Introduction}

One of the first fundamental achievements in the field of the regularity theory for weak solutions of second order linear elliptic differential equations is the existence of weak second derivatives. Indeed, let~$\Omega$ be an open set of~$\R^n$ and~$u \in H^1(\Omega)$ a weak solution of
\begin{equation} \label{2ordeq}
- \dive \left( A(\cdot) \nabla u \right) = f \quad \mbox{in } \Omega,
\end{equation}
where the~$n \times n$ matrix~$A = [a_{i j}]$ is uniformly elliptic, with entries~$a_{i j} \in C^{0, 1}_\loc(\Omega)$, and the right-hand term~$f \in L^2(\Omega)$. Then, one gets that~$u \in H^2_\loc(\Omega)$ and, for any domain~$\Omega' \subset \subset \Omega$,
$$
\| u \|_{H^2(\Omega')} \le C \left( \| u \|_{L^2(\Omega)} + \| f \|_{L^2(\Omega)} \right),
$$
for some constant~$C > 0$ independent of~$u$ and~$f$.

Such result is typically ascribed to Louis Nirenberg, who in~\cite{N55} obtained higher order Sobolev regularity for general linear elliptic equations. To do so, he introduced the by now classical~\emph{translation method}. In the setting of equation~\eqref{2ordeq} the idea is basically to consider the difference quotients
$$
D_i^h u(x) := \frac{u(x + h e_i) - u(x)}{h},
$$
for~$i = 1, \ldots, n$ and~$h \ne 0$ suitably small in modulus, and use the equation itself to recover a uniform bound in~$h$ for the gradient of~$D_i^h u$ in~$L^2(\Omega')$. A compactness argument then shows that~$u \in H^2_\loc(\Omega)$. Nice presentations of this technique are for instance contained in~\cite{E98} and~\cite{GM12}.

After this, several generalizations were achieved.
For example, the translation method has been successfully adapted to study nonlinear equations, too. Indeed, in~\cite{S77} and~\cite{D82} the authors deduced higher order regularity in both Sobolev and Besov classes for singular or degenerate operators of~$p$-Laplacian type. See also~\cite{M03,M03b} where similar fractional estimates were obtained in a non-differentiable vectorial setting. 

\medskip

The object of this note is the attempt of a generalization of the above discussed higher differentiability to a non-local analogue of equation~\eqref{2ordeq}, modelled upon the fractional Laplacian.

Given any open set~$\Omega \subset \R^n$, we consider a solution~$u$ of the linear equation
\begin{equation} \label{EKeq}
\E_K (u, \varphi) = \langle f, \varphi \rangle_{L^2(\Omega)} \quad \mbox{for any } \varphi \in C^\infty_0(\Omega),
\end{equation}
where~$f \in L^2(\Omega)$ and~$\E_K$ is defined by
$$
\E_K(u, \varphi) := \int_{\R^n} \int_{\R^n} \left( u(x) - u(y) \right) \left( \varphi(x) - \varphi(y) \right)  K(x, y) \, dx dy.
$$
Here~$K$ is a measurable function which is comparable~\emph{in the small} to the kernel of the fractional Laplacian. Indeed, if we take
$$
K(x, y) = |x - y|^{- n - 2 s},
$$
with~$s \in (0, 1)$, then~\eqref{EKeq} is the weak formulation of the equation
$$
(-\Delta)^s u = f \quad \mbox{in } \Omega,
$$
for the fractional Laplace operator of order~$2 s$ 
$$
(-\Delta)^s u(x) = 2 \, \PV \int_{\R^n} \frac{u(x) - u(y)}{|x - y|^{n + 2 s}} \, dy = 2 \lim_{\delta \rightarrow 0^+} \int_{\R^n \setminus B_\delta(x)} \frac{u(x) - u(y)}{|x - y|^{n + 2 s}} \, dy.
$$
On the other hand, more general kernels are admissible as well, possibly not translation invariant. However, if the kernel is not translation invariant, we need to impose on~$K$ some sort of~\emph{joint} local~$C^{0, s}$ regularity. We stress that this last hypothesis seems very natural to us. Indeed, while translation invariant kernels correspond in the local framework to the constant coefficient case, asking~$K$ to be locally H\"{o}lder continuous is a legitimate counterpart to the Lipschitz regularity assumed on the matrix~$A$ in~\eqref{2ordeq}.

\medskip

Integro-differential equations have been the object of a great variety of studies in recent years. A priori estimates for quite general linear equations were obtained in~\cite{BK05,S06} (H\"older estimates) and in~\cite{B09} (Schauder estimates). Other fundamental results in what concerns pointwise regularity were achieved by Caffarelli and Silvestre in~\cite{CS09, CS11}. The two authors developed there a theory for \emph{viscosity} solutions, in order to deal with general fully nonlinear equations. The framework considered here is instead that of \emph{weak} (or \emph{energy}) solutions. These two notions of solutions are of course very close, as it is discussed in~\cite{RS14} and~\cite{SV14}, but, since we have a datum~$f$ in~$L^2$, the weak formulation~\eqref{EKeq} seems to us more appropriate.

The literature on the regularity theory for weak solutions is indeed very rich and it is not possible to provide here an exhaustive account of the many contributions. Just to name a few, Kassmann addressed the validity of a Harnack inequality and established interior H\"{o}lder regularity for~\emph{non-local harmonic functions} through the language of Dirichlet forms (see~\cite{K07, K09, K11}). In~\cite{RS14} the authors obtained H\"{o}lder regularity up to the boundary for a Dirichlet problem driven by the fractional Laplacian. Concerning regularity results in Sobolev spaces,~$H^{2 s}$ estimates are proved in~\cite{DK12} for entire translation invariant equations. Also, the very recent~\cite{KMS15} provides higher differentiability/integrability in a nonlinear setting quite similar to ours.

\medskip

Here we show that a solution~$u$ of~\eqref{EKeq} has better weak (fractional) differentiability properties in the interior of~$\Omega$. By adapting the translation method to this non-local setting, we prove that
\begin{equation} \label{uN2s}
u \in N^{2 s, 2}_\loc(\Omega).
\end{equation}
Notice that the symbol~$N^{r, p}(\Omega)$, for~$r > 0$ and~$1 \le p < +\infty$, denotes here the so-called~\emph{Nikol'skii space}.

Since both Nikol'skii and fractional Sobolev spaces are part of the wider class of Besov spaces, standard embedding results within this scale allow us to deduce from~\eqref{uN2s} that
\begin{equation} \label{uH2s-eps}
u \in H^{2 s - \varepsilon}_\loc(\Omega),
\end{equation}
for any~$\varepsilon > 0$.

We do not know whether or not~\eqref{uH2s-eps} is the optimal interior regularity for solutions of~\eqref{EKeq} in the Sobolev class. While one would arguably expect~$u$ to belong to~$H^{2 s}_\loc(\Omega)$, there is no hope in general to extend such regularity up to the boundary, as discussed in Section~\ref{optsec}. Finally, we stress that the exponent~$2 s - \varepsilon$ still provides Sobolev regularity for the gradient of~$u$, when~$s > 1 / 2$.

We point out that, almost concurrently to the present work and independently from it, a result rather similar to~\eqref{uH2s-eps} has been obtained in~\cite{BL15}. Indeed, the authors address there the problem of establishing higher Sobolev regularity for a nonlinear, superquadratic generalization of equation~\eqref{EKeq}. When restricted to the linear case, their result is analogous to ours, for~$s \le 1/2$, and slightly weaker, for~$s > 1/2$.

\medskip

In the upcoming section we specify the framework in which the model is set. We give formal definitions of the notion of solution and of the class of kernels under consideration. Moreover, we introduce the various functional spaces that are necessary for these purposes. After such preliminary work, we are then in position to give the precise statements of our results.

\section{Definitions and formal statements} \label{defsec}

Let~$n \in \N$ and~$s \in (0, 1)$. The kernel~$K: \R^n \times \R^n \to [0, +\infty]$ is assumed to be measurable and symmetric\footnote{We stress that the symmetry hypothesis does not really play much of a role here. Indeed, if one considers instead a non-symmetric kernel~$K$, this can be written as the sum of its symmetric and anti-symmetric parts
$$
K_{\rm{sym}}(x, y) := \frac{K(x, y) + K(x, y)}{2} \quad \mbox{and} \quad K_{\rm{asym}}(x, y) := \frac{K(x, y) - K(y, x)}{2}.
$$
But then, it is easily shown that~$K_{\rm{asym}}$ cancels out in~\eqref{bilform}, thus leading to an equation driven by the symmetric kernel~$K_{\rm{sym}}$. Hence, we may and do assume~$K$ symmetric from the outset.

In this regard, we refer the interested reader to~\cite{FKV13}, where a class of integro-differential equations with non-symmetric kernels are studied.}, that is
\begin{equation} \label{Ksym}
K(x, y) = K(y, x) \quad \mbox{for a.a. } x, y \in \R^n.
\end{equation}
We also require~$K$ to satisfy
\begin{subequations}
\label{Kbounds}
\begin{align} \label{Kbounds1}
\lambda \le |x - y|^{n + 2 s} K(x, y) \le \Lambda & \quad \mbox{for a.a. } x, y \in \R^n, \, |x - y| < 1, \\ \label{Kbounds2}
0 \le |x - y|^{n + \beta} K(x, y) \le M & \quad \mbox{for a.a. } x, y \in \R^n, \, |x - y| \ge 1,
\end{align}
\end{subequations}
for some constants~$\Lambda \ge \lambda > 0$,~$\beta, M > 0$, and
\begin{equation} \label{Kreg}
|x - y|^{n + 2 s} \left| K(x + z, y + z) - K(x, y) \right| \le \Gamma |z|^s,
\end{equation}
for a.a.~$x, y, z \in \R^n$, with~$|x - y|, |z| < 1$, and for some~$\Gamma > 0$.

Condition~\eqref{Kbounds1} tells that the kernel~$K$ is controlled from above and below by that of the fractional Laplacian when~$x$ and~$y$ are close. Conversely, when~$|x - y|$ is large, the behaviour of~$K$ could be more general, as expressed by~\eqref{Kbounds2}.  Under these hypotheses a great variety of kernels could be encompassed, as for instance truncated ones or having non-standard decay at infinity. Naturally, these requirements are fulfilled (with~$\beta = 2 s$) when~$K$ is globally comparable to the kernel of the fractional Laplacian, that is when~\eqref{Kbounds1} holds a.e. on the whole~$\R^n \times \R^n$.

On the other hand,~\eqref{Kreg} asserts that the map
$$
(x, y) \longmapsto |x - y|^{n + 2 s} K(x, y),
$$
is locally uniformly~$C^{0, s}$ regular, jointly in the two variables~$x$ and~$y$.
Clearly,~\eqref{Kreg} is satisfied by translation invariant kernels, i.e. those in the form
\begin{equation} \label{Kti}
K(x, y) = k(x - y),
\end{equation}
for some measurable~$k: \R^n \to [0, +\infty]$. But more general choices are possible, as for instance kernels of the type
$$
K(x, y) = \frac{a(x, y)}{|x - y|^{n + 2 s}},
$$
with~$a \in C^{0, s}(\R^n \times \R^n)$. We also stress that~\eqref{Kreg} may be actually weakened by requiring it to hold only inside the set~$\Omega$ where the equation will be valid.

\medskip

In order to formulate the equation and state our main results, we introduce the following functional framework.

Let~$s > 0$,~$1 \le p < +\infty$ and~$U$ be any open set of~$\R^n$. We indicate with~$L^p(U)$ the standard Lebesgue space and with~$W^{s, p}(U)$ the (fractional) Sobolev space as defined, for instance, in the monograph~\cite{DPV12}. Of course,~$H^s(U) := W^{s, 2}(U)$.

Restricting ourselves to~$s \in (0, 1)$, we denote with~$X(U)$ the space of measurable functions~$u : \R^n \to \R$ such that
$$
u|_U \in L^2(U) \quad \mbox{and} \quad (x, y) \longmapsto \left( u(x) - u(y) \right) \sqrt{K(x, y)} \in L^2(\C_U),
$$
where
$$
\C_U := \left( \R^n \times \R^n \right) \setminus \left( \left( \R^n \setminus U \right) \times \left( \R^n \setminus U \right) \right) \subset \R^n \times \R^n.
$$
Notice that, by virtue of~\eqref{Kbounds}, if~$u \in X(U)$ and~$V$ is a bounded open set contained in~$U$, then~$u|_V \in H^s(V)$. In addition,~$X_0(U)$ is the subspace of~$X(U)$ composed by the functions which vanish a.e. outside~$U$. We refer the reader to~\cite[Section~5]{SV13} for informations on very similar spaces of functions.

As it is customary, given any space~$F(U)$ of functions defined on a set~$U$, we say that
$$
u \in F_\loc(U) \quad \mbox{if and only if} \quad u|_V \in F(V) \mbox{ for any domain } V \subset \subset U.
$$

Let now~$\Omega$ be a fixed open set of~$\R^n$. For~$u \in X(\Omega)$ and~$\varphi \in X_0(\Omega)$, it is well-defined the bilinear form
\begin{equation} \label{bilform}
\E_K(u, \varphi) := \int_{\R^n} \int_{\R^n} \left( u(x) - u(y) \right) \left( \varphi(x) - \varphi(y) \right) K(x, y) \, dx dy.
\end{equation}
Given~$f \in L^2(\Omega)$, we say that~$u \in X(\Omega)$ is a solution of
\begin{equation} \label{maineq}
\E_K(u, \cdot) = f \quad \mbox{in } \Omega,
\end{equation}
if
\begin{equation} \label{weakform}
\E_K(u, \varphi) = \langle f, \varphi \rangle_{L^2(\Omega)} \quad \mbox{for any } \varphi \in X_0(\Omega).
\end{equation}
We remark that, for instance when~$K$ is symmetric and translation invariant, i.e. as in~\eqref{Kti} with~$k$ even, then~\eqref{weakform} is the weak formulation of the equation
$$
\L_k u = f \quad \mbox{in } \Omega,
$$
where the operator~$\L_k$ is defined - for~$u$ sufficiently smooth and bounded - by
$$
\L_k u(x) := 2 \, \PV \int_{\R^n} \left( u(x) - u(y) \right) k(x - y) \, dy.
$$

As a last step towards the first theorem, we introduce a weighted Lebesgue space which we will require the solutions to lie in. Given a measurable function~$w: \R^n \to [0, +\infty)$, we say that~$u \in L^1_w(\R^n)$ if and only if
$$
u: \R^n \to \R \mbox{ is measurable} \quad \mbox{and} \quad \| u \|_{L^1_w(\R^n)} := \int_{\R^n} |u(x)| w(x) \, dx < +\infty.
$$
In what follows we consider weights of the form
\begin{equation} \label{wsdef}
w_{x_0, \beta}(x) = \frac{1}{1 + |x - x_0|^{n + \beta}},
\end{equation}
for~$x_0 \in \R^n$ and~$\beta > 0$ as in~\eqref{Kbounds2}. We denote the corresponding spaces just with~$L^1_{x_0, \beta}(\R^n)$ and we adopt the same notation for their norms. Also, we simply write~$L^1_\beta(\R^n)$ when~$x_0$ is the origin. Notice that, in fact, the space~$L^1_{x_0, \beta}(\R^n)$ does not depend on~$x_0$ and different choices for the base point~$x_0$ lead to equivalent norms. Lastly, we observe that, in consequence of the fact that~$w_{x_0, \beta} \in L^1(\R^n) \cap L^\infty(\R^n)$, the space~$L^1_\beta(\R^n)$ contains both~$L^\infty(\R^n)$ and~$L^1(\R^n)$.

\medskip

With all this in hand, we are now ready to state the first and principal result of this note.

\begin{theorem} \label{sobmainthm}
Let~$s \in (0, 1)$,~$\beta > 0$ and~$\Omega \subset \R^n$ be an open set. Assume that~$K$ satisfies assumptions~\eqref{Ksym},~\eqref{Kbounds} and~\eqref{Kreg}. Let~$u \in X(\Omega) \cap L^1_\beta(\R^n)$ be a solution of~\eqref{maineq}, with~$f \in L^2(\Omega)$. Then,~$u \in H_\loc^{2 s - \varepsilon}(\Omega)$ for any small~$\varepsilon > 0$ and, for any domain~$\Omega' \subset \subset \Omega$,
\begin{equation} \label{sobmainest}
\| u \|_{H^{2 s - \varepsilon}(\Omega')} \le C \left( \| u \|_{L^2(\Omega)} + \| u \|_{L^1_\beta(\R^n)} + \| f \|_{L^2(\Omega)} \right),
\end{equation}
for some constant~$C > 0$ depending on~$n$,~$s$,~$\beta$,~$\lambda$,~$\Lambda$,~$M$,~$\Gamma$,~$\Omega$,~$\Omega'$ and~$\varepsilon$.
\end{theorem}

The technique we adopt to prove Theorem~\ref{sobmainthm} is basically the translation method of Nirenberg, suitably adjusted to cope with the difficulties arising in this fractional, non-local framework. However, this strategy does not immediately lead to an estimate in Sobolev spaces. In fact, it provides that the solution belongs to a slightly different functional space, which is well-studied in the literature and is often referred to as~\emph{Nikol'skii space}. We briefly introduce such class here below.

Let~$U$ be a domain of~$\R^n$. Given~$k \in \N$ and~$z \in \R^n$, let
\begin{equation} \label{Ukz}
U_{k z} := \left\{ x \in U : x + i z \in U \mbox{ for any } i = 1, \ldots, k \right\}.
\end{equation}
Observe that, by definition,
\begin{equation} \label{kzinjz}
U_{k z} \subseteq U_{j z} \subseteq U \qquad \mbox{ if }  j, k \in \N \mbox{ and } j \le k.
\end{equation}
For any~$z \in \R^n$ we also define~$\tau_z u(x) := u(x + z)$ and
$$
\Deltaz u(x) := \tau_z u(x) - u(x),
$$
for any~$x \in U_z$. Sometimes we will need to deal with increments along the~\emph{diagonal} for the kernel~$K$, as previously done in~\eqref{Kreg}. With a slight abuse of notation, we write
$$
\tau_z K(x, y) := K(x + z, y + z) \quad \mbox{ and } \quad \Deltaz K(x, y) := \tau_z K(x, y) - K(x, y).
$$
We also consider increments of higher orders. For any~$k \in \N$ we set
$$
\Delta_z^k u(x) := \, \Deltaz \Delta_z^{k - 1} u(x) = \sum_{i = 0}^k (-1)^{k - i} \binom{k}{i} \tau_{i z} u(x),
$$
for any~$x \in U_{k z}$, with the convention that~$\Delta_z^0 u = u$. Of course,~$\Delta_z^1 u = \Deltaz u$. Moreover, notice that by~\eqref{kzinjz} all~$\Delta_z^j u$, as~$j = 0, 1, \ldots, k$, are well-defined in~$U_{k z}$.

Given~$s \in (0, 2)$ and~$1 \le p < +\infty$, the Nikol'skii space~$N^{s, p}(U)$ is defined as the space of functions~$u \in L^p(U)$ such that
\begin{equation} \label{niksemintro}
[u]_{N^{s, p}(U)} := \sup_{z \in \R^n \setminus \{ 0 \}} |z|^{- s} \| \Delta_z^2 u \|_{L^p(U_{2 z})} < +\infty.
\end{equation}
The norm
$$
\| u \|_{N^{s, p}(U)} := \| u \|_{L^p(U)} + [u]_{N^{s, p}(U)},
$$
makes~$N^{s, p}(U)$ a Banach space. We point out that the restriction to~$s < 2$ is assumed here only to avoid unnecessary complications in the definition of the semi-norm~\eqref{niksemintro}. By the way, the above range for~$s$ is large enough for our scopes and, thus, there is no real need to deal with more general conditions. Nevertheless, such limitation will not be considered anymore in Section~\ref{besovsec}, where a deeper look at the space~$N^{s, p}(U)$ will be given.

\medskip

Now that the definition of Nikol'skii spaces has been recalled, we may finally head to our second main result.

\begin{theorem} \label{nikmainthm}
Let~$s \in (0, 1)$,~$\beta > 0$ and~$\Omega \subset \R^n$ be an open set. Assume that~$K$ satisfies assumptions~\eqref{Ksym},~\eqref{Kbounds} and~\eqref{Kreg}. Let~$u \in X(\Omega) \cap L^1_\beta(\R^n)$ be a solution of~\eqref{maineq}, with~$f \in L^2(\Omega)$. Then,~$u \in N^{2 s, 2}_\loc(\Omega)$ and, for any domain~$\Omega' \subset \subset \Omega$,
\begin{equation} \label{nikmainest}
\| u \|_{N^{2 s, 2}(\Omega')} \le C \left( \| u \|_{L^2(\Omega)} + \| u \|_{L^1_\beta(\R^n)} + \| f \|_{L^2(\Omega)} \right),
\end{equation}
for some constant~$C > 0$ depending on~$n$,~$s$,~$\beta$,~$\lambda$,~$\Lambda$,~$M$,~$\Gamma$,~$\Omega$ and~$\Omega'$.
\end{theorem}

In light of this estimate, Theorem~\ref{sobmainthm} follows more or less immediately. To see this, it is helpful to understand Sobolev and Nikol'skii spaces in the context of~\emph{Besov spaces}. For~$s \in (0, 2)$,~$1 \le p < +\infty$ and~$1 \le \lambda \le +\infty$, the Besov space~$B_{\mini{\lambda}}^{s, p}(U)$ is the space of functions~$u \in L^p(U)$ such that~$[ u ]_{B_{\mini{\lambda}}^{s, p}(U)} < + \infty$, where
$$
[ u ]_{B_{\mini{\lambda}}^{s, p}(U)} :=
\begin{dcases}
\left( \int_{\R^n} \left( |z|^{- s} \| \Delta_z^2 u \|_{L^p(U_{2 z})} \right)^\lambda \frac{dz}{|z|^n} \right)^{1 / \lambda} & \quad \mbox{if } 1 \le \lambda < + \infty, \\
\sup_{z \in \R^n \setminus \{ 0 \}} |z|^{- s} \| \Delta_z^2 u \|_{L^p(U_{2 z})} & \quad \mbox{if } \lambda = +\infty.
\end{dcases}
$$
Observe that, by definition,~$B_{\mini{\infty}}^{s, p}(U) = N^{s, p}(U)$, while the equivalence~$B_{\mini{p}}^{s, p}(U) = W^{s, p}(U)$ is also true, though less trivial. Then, since there exist continuous embeddings
\begin{equation} \label{besinclintro}
B_{\mini{\nu}}^{s, p}(U) \subset B_{\mini{\lambda}}^{r, p}(U),
\end{equation}
as~$1 \le \lambda \le \nu \le + \infty$ and~$1 < r < s < + \infty$, it follows
$$
N^{s, p}(U) \subset W^{r, p}(U).
$$
Consequently, up to some minor details that will be discussed later in Section~\ref{sobmainsec}, Theorem~\ref{sobmainthm} is a consequence of Theorem~\ref{nikmainthm}.

Of course, Theorem~\ref{nikmainthm} and inclusion~\eqref{besinclintro} yield estimates in many other Besov spaces for the solution~$u$ of~\eqref{maineq}. Basically,~$u$ lies in any~$B_{\mini{\lambda}, \loc}^{2 s - \varepsilon, 2}(\Omega)$, with~$\varepsilon > 0$ and~$1 \le \lambda \le +\infty$.

\medskip

We point out here that throughout the paper the same letter~$c$ is used to denote a positive constant which may change from line to line and depends on the various parameters involved.

\medskip

The rest of the paper is organized as follows.

In Section~\ref{besovsec} we review some basic material on Sobolev and Nikol'skii spaces. To keep a leaner notation, we do not approach Besov spaces in their full generality and restrict in fact to the two classes to which we are interested. Despite every assertion of this section is classical and surely well-known to the experts, we choose to include here the few results that will be used afterwards, in order to make the work as self-contained as possible.

The subsequent two sections are devoted to some auxiliary results. Section~\ref{auxsec} is concerned with a couple of technical lemmata that deal with a discrete integration by parts formula and an estimate for the defect of two translated balls. In Section~\ref{caccsec}, on the other hand, we prove a non-local version of the classical Caccioppoli inequality.

The main results are proved in Sections~\ref{mainsec} and~\ref{sobmainsec}.

Finally, Section~\ref{optsec} contains some comments on the possible optimal global regularity for the Dirichlet problem associated to~\eqref{maineq}.

\section{Preliminaries on Sobolev and Nikol'skii spaces} \label{besovsec}

We collect here some general facts about fractional Sobolev spaces and Nikol'skii spaces. As said before, we avoid dealing with the wider class of Besov spaces in order not to burden the notation too much. For more complete and exhaustive presentations we refer the interested reader to the books by Triebel,~\cite{T83, T92, T06} and~\cite{T95}.

We remark that the proofs displayed only make use of integration techniques, mostly inspired by~\cite{S90}. While some results can not be justified with such elementary arguments, we still provide specific references to the above mentioned books.

\medskip

Let~$U \subset \R^n$ be a bounded domain with~$C^\infty$ boundary\footnote{Most of the assertions contained in this section should be also true under less restrictive hypotheses on the boundary of the set. Of course, the definitions of the spaces require no assumptions at all on the boundary and other results are extended in the literature to Lipschitz sets. Unfortunately, we have not been able to find completely satisfactory references for Proposition~\ref{sobaltnorm}, and its counterpart for Nikol'skii spaces, under such weaker assumptions. Anyway, the limitation to~$C^\infty$ domains will not have any influence on our applications.}. Let~$1 \le p < +\infty$ and~$s > 0$, with~$s \notin \N$. Write~$s = k + \sigma$, with~$k \in \N \cup \{ 0 \}$ and~$\sigma \in (0, 1)$. We recall that the fractional Sobolev space~$W^{s, p}(U)$ is defined as the set of functions
$$
W^{s, p}(U) := \left\{ u \in W^{k, p}(U) : [D_\alpha u]_{W^{\sigma, p}(U)} < +\infty \mbox{ for any } |\alpha| = k \right\},
$$
where, for~$v \in L^p(U)$,
$$
[v]_{W^{\sigma, p}(U)} := \left( \int_U \int_U \frac{|v(x) - v(y)|^p}{|x - y|^{n + \sigma p}} \, dx dy \right)^{1 / p}.
$$
Clearly,~$\alpha$ indicates a multi-index, i.e.~$\alpha = (\alpha_1, \ldots, \alpha_n)$ with~$\alpha_i \in \N \cup \{ 0 \}$, and~$|\alpha| = \alpha_1 + \cdots + \alpha_n$ is its modulus. Moreover,~$W^{k, p}(\Omega)$, for~$k \in \N$, denotes the standard Sobolev space and, when~$k = 0$, we understand~$W^{0, p}(U) = L^p(U)$. The space~$W^{s, p}(U)$ equipped with the norm
$$
\| u \|_{W^{s, p}(U)} := \| u \|_{W^{k , p}(U)} + \sum_{|\alpha| = k} [D^\alpha u]_{W^{\sigma, p}(U)},
$$
is a Banach space.

Notice that, for~$v \in L^p(U)$,
\begin{align*}
[v]_{W^{\sigma, p}(U)} & = \left( \int_U \int_U \frac{|v(x) - v(y)|^p}{|x - y|^{n + \sigma p}} \, dx dy \right)^{1 / p} \\
& = \left( \int_{\R^n} \left( \int_{U_z} \frac{|v(x + z) - v(x)|^p}{|z|^{n + \sigma p}} \, dx \right) dz \right)^{1 / p} \\
& = \left( \int_{\R^n} \left( |z|^{- \sigma} \| \Deltaz v \|_{L^p(U_z)} \right)^p \frac{dz}{|z|^n} \right)^{1 / p}.
\end{align*}
In view of this fact, we have the following characterization for~$W^{s, p}(U)$.

\begin{proposition} \label{sobaltnormprop}
Let~$1 \le p < +\infty$ and~$s > 0$. Let~$k, l \in \Z$ be such that~$0 \le k < s$ and~$l > s - k$. Then,
\begin{equation} \label{sobaltnorm}
\| u \|_{L^p(U)} + \sum_{|\alpha| = k} \left( \int_{\R^n} \left( |z|^{k - s} \| \Delta_z^l D^\alpha u \|_{L^p(U_{l z})} \right)^p \frac{dz}{|z|^n} \right)^{1 / p},
\end{equation}
is a Banach space norm for~$W^{s, p}(U)$, equivalent to~$\| \cdot \|_{W^{s, p}(U)}$.
\end{proposition}

A reference for these equivalences is given by Theorem~4.4.2.1 at page~323 of~\cite{T95}. Note that the result is valid even if~$s$ is an integer.

\begin{remark} \label{sobaltnormrmk}
In what follows, we will be mostly interested in norms with~$k = 0$ and therefore~$l > s$. In such cases, we stress that~\eqref{sobaltnorm} may be replaced with the \emph{restricted} norm
\begin{equation} \label{sobresnorm}
\| u \|_{L^p(U)} + \left( \int_{B_\delta} \left( |z|^{- s} \| \Delta_z^l u \|_{L^p(U_{l z})} \right)^p \frac{dz}{|z|^n} \right)^{1 / p},
\end{equation}
for any~$\delta > 0$, with no modifications to the space~$W^{s, p}(U)$. Indeed, we have
$$
\| \Delta_z^l u \|_{L^p(U_{l z})} \le 2^l \| u \|_{L^p(U)},
$$
so that
$$
\left( \int_{\R^n \setminus B_\delta} \left( |z|^{- s} \| \Delta_z^l u \|_{L^p(U_{l z})} \right)^p \frac{dz}{|z|^n} \right)^{1 / p} \le 2^l \left( \frac{\Haus^{n - 1}(\partial B_1)}{s p} \right)^{1 / p} \delta^{- s} \| u \|_{L^p(U)}.
$$
Consequently, the norms defined by~\eqref{sobaltnorm} and~\eqref{sobresnorm} are equivalent.
%
\end{remark}

The second class of fractional spaces which we are interested in are the Nikol'skii spaces. For~$s = k + \sigma > 0$, with~$k \in \N \cup \{ 0 \}, \sigma \in (0, 1]$, and~$1 \le p < +\infty$, define
$$
N^{s, p}(U) := \left\{ u \in W^{k, p}(U) : [D^\alpha u]_{N^{\sigma, p}(U)} < +\infty \mbox{ for any } |\alpha| = k \right\},
$$
where, for~$v \in L^p(U)$,
$$
[v]_{N^{\sigma, p}(U)} := \sup_{z \in \R^n \setminus \{ 0 \}} |z|^{- \sigma} \| \Delta_z^2 v \|_{L^p(U_{2 z})}.
$$
It can be showed that~$N^{s, p}(U)$ is a Banach space with respect to the norm
$$
\| u \|_{N^{s, p}(U)} := \| u \|_{W^{k, p}(U)} + [u]_{N^{s, p}(U)}.
$$

Notice that this definition of Nikol'skii space may seem to differ from that given in Section~\ref{defsec}. In fact, this is not the case, as~$N^{s, p}(U)$ can be equivalently endowed with any norm of the form
\begin{equation} \label{nikoaltnorm}
\| u \|_{L^p(U)} + \sum_{|\alpha| = k} \sup_{z \in \R^n \setminus \{ 0 \}} |z|^{k - s} \| \Delta_z^l D^\alpha u \|_{L^p(U_{l z})},
\end{equation}
where~$k, l \in \Z$ are such that~$0 \le k < s$ and~$l > s - k$ (see again Theorem~4.4.2.1 of~\cite{T95}).

\begin{remark} \label{nikaltnormrmk}
As for the Sobolev spaces, we will consider norms with~$k = 0$ for the most of the time. We stress that in such cases~\eqref{nikoaltnorm} may be replaced with
$$
\| u \|_{L^p(U)} + \sup_{0 < |z| < \delta} |z|^{- s} \| \Delta_z^l u \|_{L^p(U_{l z})},
$$
for any integer~$l > s$ and any~$\delta > 0$.
\end{remark}

\smallskip

In the conclusive part of this section we study the mutual inclusion properties of~$W^{s, p}(U)$ and~$N^{s, p}(U)$. In order to do this, it will be useful to consider another family of equivalent norms. To this aim, for~$l \in \N$ we introduce the so-called~$l$-th \emph{modulus of smoothness} of~$u$
$$
\omega_p^l(u; \eta) := \sup_{0 < |z| < \eta} \| \Delta_z^l u \|_{L^p(U_{l z})},
$$
defined for any~$\eta > 0$. Then, we have

\begin{proposition} \label{2altnormprop}
Let~$s > 0$ and~$1 \le p < +\infty$. Let~$l > s$ be an integer and~$0 < \delta \le +\infty$. Then,
$$
\| u \|_{L^p(U)} + \left( \int_0^\delta \left( \eta^{- s} \omega_p^l(u; \eta) \right)^p \, \frac{d\eta}{\eta} \right)^{1 / p},
$$
is a Banach space norm for~$W^{s, p}(U)$, equivalent to~$\| \cdot \|_{W^{s, p}(U)}$.

The same statement holds true for the norms
$$
\| u \|_{L^p(U)} + \sup_{0 < \eta < \delta} \eta^{- s} \omega_p^l(u; \eta),
$$
and the space~$N^{s, p}(U)$.
\end{proposition}

\begin{proof}
We only deal with the Sobolev space case, the Nikol'skii one being completely analogous and easier. Furthermore, we assume~$\delta = 1$. Then, an argument similar to that presented in Remark~\ref{sobaltnormrmk} shows that the result can be extended to any~$\delta$.

For~$u \in L^p(U)$ let
$$
[u]_{W^{s, p}(U)}^\flat := \left( \int_{B_1} \left( |z|^{- s} \| \Delta_z^l u \|_{L^p(U_{l z})} \right)^p \, \frac{dz}{|z|^n} \right)^{1 / p},
$$
and
$$
[u]_{W^{s, p}(U)}^\sharp := \left( \int_0^1 \left( \eta^{- s} \omega_p^l(u; \eta) \right)^p \, \frac{d\eta}{\eta} \right)^{1 / p}.
$$
We claim that there exists a constant~$c \ge 1$ such that
\begin{equation} \label{Wnormeq}
c^{- 1} [u]_{W^{s, p}(U)}^\flat \le [u]_{W^{s, p}(U)}^\sharp \le c \left( \| u \|_{L^p(U)} + [u]_{W^{s, p}(U)}^\flat \right),
\end{equation}
for all~$u \in L^p(U)$. In view of Proposition~\ref{sobaltnormprop} and Remark~\ref{sobaltnormrmk}, this concludes the proof.

To check the left hand inequality of~\eqref{Wnormeq} we first observe that
$$
\| \Delta_z^l u \|_{L^p(U_{l z})} \le \sup_{0 < |y| < |z|} \| \Delta_y^l u \|_{L^p(U_{l y})} = \omega_p^l(u; |z|),
$$
for any~$z \in \R^n$. Then, using polar coordinates,
\begin{align*}
[u]_{W^{s, p}(U)}^\flat & = \left( \int_{B_1} \left( |z|^{- s} \| \Delta_z^l u \|_{L^p(U_{l z})} \right)^p \, \frac{dz}{|z|^n} \right)^{1 / p} \\
& \le \left( \Haus^{n - 1}(\partial B_1) \int_0^1 \left( \eta^{- s} \omega_p^l(u; \eta) \right)^p \, \frac{d\eta}{\eta} \right)^{1 / p} \\
& = \Haus^{n - 1}(\partial B_1)^{1 / p} \, [u]_{W^{s, p}(U)}^\sharp.
\end{align*}

Now we focus on the second inequality. In order to show its validity we need the following auxiliary result. For~$x \in U$,~$\eta > 0$ and~$u \in L^p(U)$, let
\begin{align*}
V^l(x, \eta) & := \left\{ z \in B_\eta : x + \tau z \in U, \, \mbox{for any } 0 \le \tau \le l \right\}, \\
M_\eta^l u (x) & := \eta^{- n} \int_{V^l(x, \eta)} |\Delta_z^l u(x)|\, dz,
\end{align*}
and define
\begin{align} \label{starsemi}
[u]_{W^{s, p}(U)}^* & := 
\left( \int_0^1 \left( \eta^{- s} \| M_\eta^l u \|_{L^p(U)} \right)^p \, \frac{d\eta}{\eta} \right)^{1 / p}, \\ \nonumber
\| u \|_{W^{s, p}(U)}^* & := \| u \|_{L^p(U)} + [u]_{W^{s, p}(U)}^*.
\end{align}
Then, by virtue of~\cite[Theorem~1.118]{T06} we infer that
\begin{equation} \label{triebelIII}
[u]_{W^{s, p}(U)}^\sharp \le c \| u \|_{W^{s, p}(U)}^*,
\end{equation}
for any~$u \in L^p(U)$.

Applying the generalized Minkowski's inequality to the right-hand side of~\eqref{starsemi} and observing that
$$
\{ (x, z) \in U \times \R^n : z \in V^l(x, \eta)\} \subseteq \left\{ (x, z) \in U \times B_\eta : x \in U_{l z} \right\},
$$
we get
\begin{equation} \label{tech}
\begin{aligned}
[u]_{W^{s, p}(U)}^* & = \left( \int_0^1 \eta^{- (s + n) p} \left( \int_U \left( \int_{V^l(x, \eta)} |\Delta_z^l u(x)| \, dz \right)^p dx \right) \frac{d\eta}{\eta} \right)^{1 / p} \\
& \le \left( \int_0^1 \eta^{- (s + n) p} \left( \int_{B_\eta} \| \Delta_z^l u \|_{L^p(U_{l z})} \, dz \right)^p \frac{d\eta}{\eta} \right)^{1 / p}.
\end{aligned}
\end{equation}
Now, Jensen's inequality implies that
$$
\left( \int_{B_\eta} \| \Delta_z^l u \|_{L^p(U_{l z})} \, dz \right)^p \le c \, \eta^{n (p - 1)} \int_{B_\eta} \| \Delta_z^l u \|_{L^p(U_{l z})}^p \, dz,
$$
and hence~\eqref{tech} becomes
$$
[u]_{W^{s, p}(U)}^* \le c \left( \int_0^1 \eta^{- n - 1 - s p} \left( \int_{B_\eta} \| \Delta_z^l u \|_{L^p(U_{l z})}^p \, dz \right) d\eta \right)^{1 / p}.
$$
We finally switch to polar coordinates to compute
\begin{align*}
[u]_{W^{s, p}(U)}^* & \le c \left( \int_0^1 \int_0^\eta \eta^{- n - 1 - s p} \left( \int_{\partial B_\rho} \| \Delta_z^l u \|_{L^p(U_{l z})}^p \, d\Haus^{n - 1}(z) \right) d\rho \, d\eta \right)^{1 / p} \\
& = c \left( \int_0^1 \left( \int_{\partial B_\rho} \| \Delta_z^l u \|_{L^p(U_{l z})}^p \, d\Haus^{n - 1}(z) \right) \left( \int_\rho^1 \eta^{- n - 1 - s p} \, d\eta \right) d\rho \right)^{1 / p} \\
& \le c \left( \int_0^1 \left( \int_{\partial B_\rho} \| \Delta_z^l u \|_{L^p(U_{l z})}^p \, d\Haus^{n - 1}(z) \right) \rho^{- n - s p} \, d\rho \right)^{1 / p} \\
& = c [u]_{W^{s, p}(U)}^\flat.
\end{align*}
By combining this formula with~\eqref{triebelIII}, we obtain the right inequality of~\eqref{Wnormeq}. Thus, the proof of the proposition is complete.
\end{proof}

\medskip

We are now in position to prove the main results of this section, concerning the relation between Sobolev and Nikol'skii spaces. First, we have

\begin{proposition} \label{sobnik1}
Let~$s > 0$ and~$1 \le p < +\infty$. Then,
$$
W^{s, p}(U) \subseteq N^{s, p}(U),
$$
and there exists a constant~$C > 0$, depending on~$n$,~$s$ and~$p$, such that
$$
\| u \|_{N^{s, p}(U)} \le C \| u \|_{W^{s, p}(U)},
$$
for any~$u \in L^p(U)$.
\end{proposition}

\begin{proof}
In view of Proposition~\ref{2altnormprop} it is enough to prove that, if~$l \in \Z$ is such that~$l > s$, then
\begin{equation} \label{tech2}
\sup_{\eta > 0} \eta^{- s} \omega_p^l(u; \eta) \le c \left( \int_0^{+\infty} \left( \eta^{- s} \omega_p^l(u; \eta) \right)^p \frac{d\eta}{\eta} \right)^{1/ p},
\end{equation}
for some~$c > 0$. But this is in turn an immediate consequence of the monotonicity of~$\omega_p^l(u; \cdot)$. Indeed,~$\omega_p^l(u; \eta) \ge \omega_p^l(u; t)$, for any~$\eta \ge t$, and so
\begin{align*}
\left( \int_0^{+\infty} \left( \eta^{- s} \omega_p^l(u; \eta) \right)^p \frac{d\eta}{\eta} \right)^{1/ p} & \ge \left( \int_t^{+\infty} \left( \eta^{- s} \omega_p^l(u; t) \right)^p \frac{d\eta}{\eta} \right)^{1/ p} = (s p)^{- 1 / p} t^{- s} \omega_p^l(u; t).
\end{align*}
Inequality~\eqref{tech2} is then obtained by taking the supremum as~$t > 0$ on the right hand side.
\end{proof}

The following provides a partial converse to the above inclusion.

\begin{proposition} \label{sobnik2}
Let~$s > r > 0$ and~$1 \le p < +\infty$. Then,
$$
N^{s, p}(U) \subseteq W^{r, p}(U),
$$
and there exists a constant~$C > 0$, depending on~$n$,~$r$,~$s$ and~$p$, such that
$$
\| u \|_{W^{r, p}(U)} \le C \| u \|_{N^{s, p}(U)},
$$
for any~$u \in L^p(U)$.
\end{proposition}

\begin{proof}
The result follows by noticing that, for~$l \in \Z$ with~$l > s$,
\begin{align*}
\left( \int_0^1 \left( \eta^{- r} \omega_p^l(u; \eta) \right)^p \frac{d\eta}{\eta} \right)^{1 / p} & = \left( \int_0^1 \eta^{(s - r)p } \left( \eta^{- s} \omega_p^l(u; \eta) \right)^p \frac{d\eta}{\eta} \right)^{1 / p} \\
& \le [(s - r) p]^{- 1 / p} \sup_{0 < \eta < 1} \eta^{- s} \omega_p^l(u; \eta),
\end{align*}
for any~$u \in L^p(U)$, and recalling Proposition~\ref{2altnormprop}.
\end{proof}

\section{Some auxiliary results} \label{auxsec}

Before we can proceed to Sections~\ref{caccsec} and~\ref{mainsec}, which contain the core argumentations leading to Theorem~\ref{nikmainthm}, we need to prove a couple of subsidiary result. 

First, we prove the following discrete version of the standard integration by parts formula.

\begin{lemma} \label{discintpartlem}
Let~$B_R$ be some ball of radius~$R > 0$ in~$\R^n$. Assume that~$K$ satisfies assumptions~\eqref{Ksym} and~\eqref{Kbounds}. Let~$u, v \in H^s(B_{8 R})$, with~$v$ supported in~$B_{2 R}$. Then,
\begin{equation} \label{discintpart}
\begin{aligned}
& \int_{B_{8 R}} \int_{B_{8 R}} \left( u(x) - u(y) \right) \left(\Delta_{-z}^2 v(x) - \Delta_{-z}^2 v(y) \right) K(x, y) \, dx dy \\
& \hspace{20pt} = \int_{B_{6 R}} \int_{B_{6 R}}  \left( \Delta_z^2 u(x) - \Delta_z^2 u(y) \right) \left( v(x) - v(y) \right) K(x, y) \, dx dy \\
& \hspace{20pt} \quad + \sum_{i = 1}^2 (-1)^i \binom{2}{i} \int_{B_{6 R}} \int_{B_{6 R}} \left( \tau_{i z} u(x) - \tau_{i z} u(y) \right) \left(v(x) - v(y) \right) \tensor*[]{\Delta}{_{i z}} K(x, y) \, dx dy \\
& \hspace{20pt} \quad - 2 \sum_{i = 0}^2 (-1)^i \binom{2}{i} \int_{B_{8 R}} \int_{B_{8 R}} \left( u(x) - u(y) \right) \tau_{- i z} \chi_{\R^n \setminus B_{6 R}}(x) \tau_{- i z} v(y) \\
& \hspace{300pt} \times K(x, y) \, dx dy,
\end{aligned}
\end{equation}
for any~$z \in \R^n$ such that~$|z| < R$.
\end{lemma}
\begin{proof}
We first expand the integral on the left hand side of~\eqref{discintpart}, obtaining
\begin{equation} \label{discpartball1}
\begin{aligned}
& \int_{B_{8 R}} \int_{B_{8 R}} \left( u(x) - u(y) \right) \left( \Delta_{-z}^2 v(x) - \Delta_{-z}^2 v(y) \right) K(x, y) \, dx dy \\
& \hspace{10pt} = \sum_{i = 0}^2 (-1)^i \binom{2}{i} \int_{B_{8 R}} \int_{B_{8 R}} \left( u(x) - u(y) \right) \left(v(x - i z) - v(y - i z) \right) K(x, y) \, dx dy.
\end{aligned}
\end{equation}
Then, we write each term on the right hand side of~\eqref{discpartball1} as\footnote{The symbol~$D + z$, where~$D$ is a set and~$z$ a vector of~$\R^n$, identifies, as conventional, the set
$$
\left\{ y \in \R^n : y = x + z \mbox{ with } x \in D \right\}.
$$
In the following formulae it is applied with~$D$ an Euclidean ball~$B_r$. Also, it should not be confused with the notation~$(B_r)_z$, which will be used later on in Section~\ref{mainsec} and has to be understood in the sense of definition~\eqref{Ukz}.}
\begin{equation} \label{discpartball2}
\begin{aligned}
& \int_{B_{8 R}} \int_{B_{8 R}} \left( u(x) - u(y) \right) \left( v(x - i z) - v(y - i z) \right) K(x, y) \, dx dy \\
& \hspace{50pt} = \int_{B_{6 R} + i z} \int_{B_{6 R} + i z} \left( u(x) - u(y) \right) \left(v(x - i z) - v(y - i z) \right) K(x, y) \, dx dy \\
& \hspace{50pt} \quad - 2 \int_{B_{8 R}} \int_{B_{8 R}} \left( u(x) - u(y) \right) \chi_{\R^n \setminus (B_{6 R} + i z)}(x) v(y - i z) K(x, y) \, dx dy.
\end{aligned}
\end{equation}
We apply the change of variables~$\tilde{x} := x - i z$,~$\tilde{y} := y - i z$ in the first integral, to get
\begin{equation} \label{discpartball2a}
\begin{aligned}
& \int_{B_{6 R} + i z} \int_{B_{6 R} + i z} \left( u(x) - u(y) \right) \left(v(x - i z) - v(y - i z) \right) K(x, y) \, dx dy \\
& \hspace{20pt} = \int_{B_{6 R}} \int_{B_{6 R}} \left( u(\tilde{x} + i z) - u(\tilde{y} + i z) \right) \left(v(\tilde{x}) - v(\tilde{y}) \right) K(\tilde{x} + i z, \tilde{y} + i z) \, d\tilde{x} d\tilde{y}.
\end{aligned}
\end{equation}
Writing then for~$i = 1, 2$
$$
K(\tilde{x} + i z, \tilde{y} + i z) = K(\tilde{x}, \tilde{y}) + \tensor*[]{\Delta}{_{i z}} K(\tilde{x}, \tilde{y}),
$$
and relabeling the variables~$\tilde{x}, \tilde{y}$ as~$x, y$, formula~\eqref{discpartball2a} becomes
\begin{equation} \label{discpartball2b}
\begin{aligned}
& \int_{B_{6 R} + i z} \int_{B_{6 R} + i z} \left( u(x) - u(y) \right) \left(v(x - i z) - v(y - i z) \right) K(x, y) \, dx dy \\
& \hspace{50pt} = \int_{B_{6 R}} \int_{B_{6 R}} \left( u(x + i z) - u(y + i z) \right) \left(v(x) - v(y) \right) K(x, y) \, dx dy \\
& \hspace{50pt} \quad + \int_{B_{6 R}} \int_{B_{6 R}} \left( u(x + i z) - u(y + i z) \right) \left(v(x) - v(y) \right) \tensor*[]{\Delta}{_{i z}} K(x, y) \, dx dy.
\end{aligned}
\end{equation}
By using~\eqref{discpartball2},~\eqref{discpartball2b} in~\eqref{discpartball1} and noticing that~$\tau_{- i z} \chi_{\R^n \setminus B_{6 R}} = \chi_{\R^n \setminus (B_{6 R} + i z)}$, we finally obtain~\eqref{discintpart}.
\end{proof}

Then, we have the following result, in which we deduce an upper bound for the measure of the symmetric difference of two translated balls in terms of the modulus of the displacement vector. Despite the estimate is almost immediate, we include a proof of it for completeness.

We also refer to~\cite{S10} for a refined version of this result, holding for general bounded sets.

\begin{lemma} \label{defectlem}
Let~$B_R$ be some ball of radius~$R > 0$ in~$\R^n$. Then, for any~$z \in \R^n$,
$$
|B_R \Delta (B_R + z)| \le C R^{n - 1} |z|,
$$
where~$C > 0$ is a dimensional constant.
\end{lemma}
\begin{proof}
First, we observe that we may restrict ourselves to~$|z| \le R / 2$, being the opposite case trivial. With the change of variables~$y := x / R$, we scale
$$
|B_R \Delta (B_R + z)| = 2 \int_{B_R \setminus (B_R + z)} dx = 2 R^n \int_{B_1 \setminus (B_1 + \hat{z})} dy,
$$
where~$\hat{z} = z / R$. Then, we easily check that
$$
B_{1 - |\hat{z}|} \subset B_1 + \hat{z},
$$
to obtain
$$
|B_R \Delta (B_R + z)| \le 2 R^n \int_{B_1 \setminus B_{1 - |\hat{z}|}} dy = \frac{2 \Haus^{n - 1}(\partial B_1)}{n} R^n \left[ 1 - (1 - |\hat{z}|)^n \right].
$$
The result then follows, since~$1 - (1 - t)^n \le n t$, for any~$t \ge 0$.
\end{proof}

\section{A Caccioppoli-type inequality} \label{caccsec}

In this section we present an estimate for the~$H^s$ norm of a solution~$u$ of~\eqref{maineq} reminiscent of the classical one by Caccioppoli. Results of this kind are by now well established also for non-local equations, for instance in~\cite{KMS15, DKP14, BP14}.

\begin{proposition} \label{caccioppoliprop}
Let~$s \in (0, 1)$,~$\beta > 0$ and~$\Omega \subset \R^n$ be an open set. Fix a point~$x_0 \in \Omega$ and let~$r > 0$ be such that~$B_r(x_0) \subset \subset \Omega$. Assume that~$K$ satisfies assumptions~\eqref{Ksym} and~\eqref{Kbounds}. Let~$u \in X(\Omega) \cap L^1_\beta(\R^n)$ be a solution of~\eqref{maineq}, with~$f \in L^2(\Omega)$. Then,
\begin{equation} \label{caccioppoli}
[ u ]_{H^s(B_r(x_0))} \le C \left( \| u \|_{L^2(\Omega)} + \| u \|_{L^1_{x_0, \beta}(\R^n)} + \| f \|_{L^2(\Omega)} \right),
\end{equation}
for some constant~$C > 0$ depending on~$n$,~$s$,~$\beta$,~$\lambda$,~$\Lambda$,~$M$,~$r$ and~$\dist \left( B_r(x_0), \partial \Omega \right)$.
\end{proposition}

We stress that hypothesis~\eqref{Kreg} is not assumed here. Consequently, Proposition~\ref{caccioppoliprop} holds for a general measurable~$K$ which only satisfies~\eqref{Kbounds}.

\begin{proof}[Proof of Proposition~\ref{caccioppoliprop}]
Our argument follows the lines of those contained in the above mentioned papers. Anyway, we provide all the details for the reader's convenience.

First, observe that we may assume~$r < 1 / 2$ for the beginning. The case of a general radius~$r > 0$ will then follow by a covering argument. Take~$R > 0$ in such a way that~$r < R < 1 / 2$ and~$B_R(x_0) \subset \Omega$. To simplify the notation, we write~$B_\rho$ instead of~$B_\rho(x_0)$, for any~$\rho > 0$.

Let~$\eta \in C^\infty_0(\R^n)$ be a cut-off function such that
\begin{equation} \label{etacutoff5}
\begin{cases}
\supp \eta \subset B_{(R + r) / 2} & \\
                  0 \le \eta \le 1 & \quad \mbox{in } \R^n \\
                          \eta = 1 & \quad \mbox{in } B_r \\
     |\nabla \eta| \le 4 / (R - r) & \quad \mbox{in } \R^n.
\end{cases}
\end{equation}

Testing~\eqref{weakform} with~$\varphi := \eta^2 u \in X_0(\Omega)$ we get
\begin{equation} \label{plug5}
\begin{aligned}
& \int_{B_R} f(x) \eta^2(x) u(x) \, dx \\
& \hspace{30pt} = \int_{B_R} \int_{B_R} \left( u(x) - u(y) \right) \left( \eta^2(x) u(x) - \eta^2(y) u(y) \right) K(x, y) \, dx dy \\
& \hspace{30pt} \quad - 2 \int_{\R^n \setminus B_R} \int_{B_R} \left( u(x) - u(y) \right) \eta^2(y) u(y) K(x, y) \, dx dy \\
& \hspace{30pt} =: I - 2 J.
\end{aligned}
\end{equation}

We estimate~$I$. Notice that
\begin{align*}
\left( u(x) - u(y) \right) & \left( \eta^2(x) u(x) - \eta^2(y) u(y) \right) \\
& \hspace{30pt} = \eta^2(x) u^2(x) - \eta^2(x) u(x) u(y) - \eta^2(y) u(x) u(y) + \eta^2(y) u^2(y) \\
& \hspace{30pt} = {\left| \eta(x) u(x) - \eta(y) u(y) \right|}^2 - {\left| \eta(x) - \eta(y) \right|}^2 u(x) u(y) \\
& \hspace{30pt} \ge {\left| \eta(x) u(x) - \eta(y) u(y) \right|}^2 - {\left| \eta(x) - \eta(y) \right|}^2 |u(x)| |u(y)|,
\end{align*}
and, therefore, using~\eqref{Kbounds1},
\begin{equation} \label{I5}
\begin{aligned}
I & \ge \lambda \int_{B_R} \int_{B_R} \frac{{\left| \eta(x) u(x) - \eta(y) u(y) \right|}^2}{|x - y|^{n + 2 s}} \, dx dy \\
& \quad - \Lambda \int_{B_R} \int_{B_R} \frac{{\left| \eta(x) - \eta(y) \right|}^2 |u(x)| |u(y)|}{|x - y|^{n + 2 s}} \, dx dy.
\end{aligned}
\end{equation}
Applying~\eqref{etacutoff5} and Young's inequality, we deduce
\begin{align*}
\int_{B_R} \int_{B_R} \frac{{\left| \eta(x) - \eta(y) \right|}^2 |u(x)| |u(y)|}{|x - y|^{n + 2 s}} \, dx dy & \le \frac{16}{(R - r)^2} \int_{B_R} \int_{B_R} \frac{|u(x)| |u(y)|}{|x - y|^{n + 2 s - 2}} \, dx dy \\
& \le \frac{16}{(R - r)^2} \int_{B_R} \int_{B_R} \frac{|u(x)|^2}{|x - y|^{n + 2 s - 2}} \, dx dy \\
& \le c \|  u \|_{L^2(B_R)}^2,
\end{align*}
which, together with~\eqref{I5}, leads to
\begin{equation} \label{Ib5}
I \ge \lambda [\eta u]_{H^s(B_R)}^2 - c \| u \|_{L^2(B_R)}^2.
\end{equation}

We now deal with~$J$. Let~$x \in \R^n \setminus B_R$ and~$y \in B_{(R + r) / 2}$. Then,
$$
|y - x_0| \le \frac{R + r}{2} \le \frac{R + r}{2 R} |x - x_0|,
$$
and so
$$
|x - y| \ge |x - x_0| - |y - x_0| \ge \frac{R - r}{2 R} |x - x_0| \ge \frac{R - r}{4} \left( 1 + |x - x_0| \right),
$$
since~$R < 1$. In view of this and~\eqref{Kbounds} we have
\begin{equation} \label{Kboundsded}
K(x, y) \le \Lambda \frac{\chi_{[0, 1)}(|x - y|)}{|x - y|^{n + 2 s}} + M \frac{\chi_{[1, +\infty)}(|x - y|)}{|x - y|^{n + \beta}}
\le \frac{c}{1 + |x - x_0|^{n + \beta}}.
\end{equation}
Moreover, using~\eqref{etacutoff5} we write
$$
|u(x) - u(y)| |u(y)| \eta^2(y) \le |u(x)| |u(y)| + |u(y)|^2,
$$
and hence by~\eqref{Kboundsded} and Young's inequality we get
\begin{equation} \label{J5}
\begin{aligned}
|J| & \le c \int_{\R^n \setminus B_R} \left( \int_{B_{(R + r) / 2}} \frac{|u(x) - u(y)| |u(y)| \eta^2(y)}{1 + |x - x_0|^{n + \beta}} \, dy \right) dx \\
& \le c \left[ \int_{B_{(R + r) / 2}} |u(y)|^2 \, dy + \left( \int_{\R^n \setminus B_R} \frac{|u(x)|}{1 + |x - x_0|^{n + \beta}} \, dx \right)^2 \right] \\
& \le c \left( \| u \|_{L^2(B_R)}^2 + \| u \|_{L^1_{x_0, \beta}(\R^n)}^2 \right).
\end{aligned}
\end{equation}

Finally, we easily compute
\begin{equation} \label{f5}
\left| \int_{B_R} f(x) u(x) \eta^2(x) \, dx \right| \le \frac{1}{2} \left( \| u \|_{L^2(B_R)}^2 + \| f \|_{L^2(\Omega)}^2 \right).
\end{equation}

Putting~\eqref{plug5},~\eqref{Ib5},~\eqref{J5} and~\eqref{f5} together, we obtain
$$
[u]_{H^s(B_r)} \le [\eta u]_{H^s(B_R)} \le c \left( \| u \|_{L^2(\Omega)} + \| u \|_{L^1_{x_0, \beta}(\R^n)} + \| f \|_{L^2(\Omega)} \right),
$$
where the first inequality follows from~\eqref{etacutoff5}. Thus,~\eqref{caccioppoli} is proved.
\end{proof}

\section{Proof of Theorem~\ref{nikmainthm}} \label{mainsec}

We are finally in position to proceed with the demonstration of our principal contribution.

\begin{proof}[Proof of Theorem~\ref{nikmainthm}]
Let~$x_0 \in \Omega$ and~$R \in (0, 1 / 56)$ be such that~$B_{56 R}(x_0) \subset \subset \Omega$. In the following any ball~$B_r$ will always be assumed to be centered at~$x_0$. Let~$\eta \in C^\infty_0(\R^n)$ be a cut-off function satisfying
\begin{equation} \label{etacutoff}
\begin{cases}
\supp \eta \subset B_{2 R} & \\
          0 \le \eta \le 1 & \quad \mbox{in } \R^n \\
                  \eta = 1 & \quad \mbox{in } B_R \\
   |\nabla \eta| \le 2 / R & \quad \mbox{in } \R^n.
\end{cases}
\end{equation}

Fix~$z \in \R^n$, with~$|z| < R$, and plug~$\varphi := \Delta_{-z}^2 \left( \eta^2 \Delta_z^2 u \right) \in X_0(\Omega)$ in formulation~\eqref{weakform}. Writing~$U = \Delta_z^2 u$, we have
\begin{equation} \label{plug}
\begin{aligned}
& \int_{B_{3 R}} f(x) \Delta_{-z}^2 \left( \eta^2 U \right)(x) \, dx \\
& \hspace{30pt} = \int_{B_{8 R}} \int_{B_{8 R}} \left( u(x) - u(y) \right) \left( \Delta_{-z}^2 \left( \eta^2 U \right)(x) - \Delta_{-z}^2 \left( \eta^2 U \right)(y) \right) K(x, y) \, dx dy \\
& \hspace{30pt} \quad - 2 \int_{\R^n \setminus B_{8 R}} \int_{B_{8 R}} \left( u(x) - u(y) \right) \Delta_{-z}^2 \left( \eta^2 U \right)(y) K(x, y) \, dy dx \\
& \hspace{30pt} =: I - 2 J.
\end{aligned}
\end{equation}

We apply Lemma~\ref{discintpartlem} to~$I$ with~$v = \eta^2 U$, obtaining
\begin{equation} \label{II1I2}
\begin{aligned}
I & = \int_{B_{6 R}} \int_{B_{6 R}} \left( U(x) - U(y) \right) \left(\eta^2(x) U(x) - \eta^2(y) U(y) \right) K(x, y) \, dx dy \\
& \quad + \sum_{i = 1}^2 (-1)^i \binom{2}{i} \int_{B_{6 R}} \int_{B_{6 R}} \left( \tau_{i z} u(x) - \tau_{i z} u(y) \right) \left( \left( \eta^2 U \right)(x) - \left( \eta^2 U \right)(y) \right) \\
& \hspace{300pt} \times \tensor*[]{\Delta}{_{i z}} K(x, y) \, dx dy \\
& \quad - 2 \sum_{i = 0}^2 (-1)^i \binom{2}{i} \int_{B_{8 R}} \int_{B_{8 R}} \left( u(x) - u(y) \right) \left( \tau_{- i z} \chi_{\R^n \setminus B_{6 R}}(x) \tau_{- i z} \left( \eta^2 U \right)(y) \right) \\
& \hspace{300pt} \times K(x, y) \, dx dy \\
& =: I_1 + I_2 - 2 I_3.
\end{aligned}
\end{equation}
Arguing as we did to obtain~\eqref{I5} in Proposition~\ref{caccioppoliprop}, we recover
\begin{equation} \label{I1}
I_1 \ge \lambda [\eta \Delta_z^2 u]_{H^s(B_{6 R})}^2 - c \| \Delta_z^2 u \|_{L^2(B_{6 R})}^2.
\end{equation}
The term~$I_2$ can be dealt with as follows. Applying~\eqref{Kreg} together with Young's inequality, we have
\begin{align*}
|I_2| & \le 2 \Gamma |z|^s \sum_{i = 1}^2 \int_{B_{6 R}} \int_{B_{6 R}} \frac{\left| \tau_{i z} u(x) - \tau_{i z} u(y) \right| \left| \left( \eta^2 U \right)(x) - \left( \eta^2 U \right)(y) \right|}{|x - y|^{n + 2 s}} \, dx dy \\
& \le c |z|^s \left( \delta [u]_{H^s(B_{8 R})}^2 + \delta^{- 1} [\eta^2 \Delta_z^2 u]_{H^s(B_{6 R})}^2 \right),
\end{align*}
with~$\delta > 0$. Taking~$\delta = \varepsilon^{- 2} |z|^s$, for some small~$\varepsilon > 0$, we get
\begin{equation} \label{I2}
|I_2| \le c \left( \varepsilon^{- 2} |z|^{2 s} [u]_{H^s(B_{8 R})}^2 + \varepsilon^2 [\eta^2 \Delta_z^2 u]_{H^s(B_{6 R})}^2 \right).
\end{equation}
We now estimate~$I_3$. By adding and subtracting the terms~$\tau_{- 2 z} \chi_{\R^n \setminus B_{6 R}} (x) \tau_{-z} (\eta^2 U)(y)$ and~$\tau_{-z} \chi_{\R^n \setminus B_{6 R}}(x) (\eta^2 U)(y)$, we see that
\begin{align*}
I_3 & = \sum_{i = 0}^1 \int_{B_{8 R}} \int_{B_{8 R}} \left( u(x) - u(y) \right) \tau_{- (i + 1) z} \chi_{\R^n \setminus B_{6 R}}(x) \Deltamz (\eta^2 U)(y - i z) K(x, y) \, dx dy \\
& \quad - \sum_{i = 0}^1 \int_{B_{8 R}} \int_{B_{8 R}} \left( u(x) - u(y) \right) \Deltamz \chi_{\R^n \setminus B_{6 R}}(x - i z) \tau_{- i z} (\eta^2 U)(y) K(x, y) \, dx dy \\
& =: I_3^{(1)} - I_3^{(2)}.
\end{align*}
On the one hand, using~\eqref{Kbounds1} and again the weighted Young's inequality,
\begin{align*}
\left| I_3^{(1)} \right| & \le \Lambda \sum_{i = 0}^1 \int_{B_{3 R} + i z} |\Deltamz \left( \eta^2 U \right)(y - i z)| \left( \int_{B_{8 R} \setminus \left( B_{6 R} + (i + 1) z \right)} \frac{|u(x)| + |u(y)|}{|x - y|^{n + 2 s}} \, dx \right) dy \\
& \le c \left( \delta \| u \|_{L^2(B_{8 R})}^2 + \delta^{-1} \| \Deltamz \left( \eta^2 \Delta_z^2 u \right) \|_{L^2(B_{3 R})}^2 \right).
\end{align*}
On the other hand
$$
\left| \Deltamz \chi_{\R^n \setminus B_{6 R}}(x - i z) \right| = \chi_{\left( B_{6 R} + (i + 1) z \right) \Delta \left( B_{6 R} + i z \right)}(x),
$$
and hence
\begin{align*}
\left| I_3^{(2)} \right| & \le \Lambda \sum_{i = 0}^1 \int_{B_{2 R} + i z} \left| \eta^2(y) U(y) \right| \left( \int_{\left( B_{6 R} + (i + 1) z \right) \Delta \left( B_{6 R} + i z \right)} \frac{|u(x)| + |u(y)|}{|x - y|^{n + 2 s}} \, dx \right) dy \\
& \le c \left( \gamma \left| \left( B_{6 R} + z \right) \Delta B_{6 R} \right| \| u \|_{L^2(B_{8 R})}^2 + \gamma^{-1} \| \Delta_z^2 u \|_{L^2(B_{3 R})}^2 \right),
\end{align*}
for any~$\gamma > 0$. In view of Lemma~\ref{defectlem} we have
$$
\left| \left( B_{6 R} + z \right) \Delta B_{6 R} \right| \le c |z|.
$$
Therefore,
$$
\left| I_3^{(2)} \right| \le c \left( \gamma |z| \| u \|_{L^2(B_{8 R})}^2 + \gamma^{-1} \| \Delta_z^2 u \|_{L^2(B_{3 R})}^2 \right).
$$
The choices~$\delta = \varepsilon^{-2} |z|^{2 s}$ and~$\gamma = |z|^{2 \sigma - 1}$, for some
\begin{equation} \label{sigmabound}
\sigma \ge \max \left\{ s, 1 / 2 \right\},
\end{equation}
then yield
\begin{equation*}
|I_3| \le c \left[ \varepsilon^{- 2} |z|^{2 s} \| u \|_{L^2(B_{8 R})}^2 + \varepsilon^2 |z|^{- 2 s} \| \Deltamz \left( \eta^2 \Delta_z^2 u \right) \|_{L^2(B_{3 R})}^2 + |z|^{1 - 2 \sigma} \| \Delta_z^2 u \|_{L^2(B_{3 R})}^2 \right].
\end{equation*}
By combining~\eqref{I1} and~\eqref{I2} with the above inequality, recalling~\eqref{II1I2} and~\eqref{sigmabound} we get
\begin{equation} \label{I}
\begin{aligned}
I & \ge \lambda [\eta \Delta_z^2 u]_{H^s(B_{6 R})}^2 - c \Big[ \varepsilon^{- 2} |z|^{2 s} \| u \|_{H^s(B_{8 R})}^2  + |z|^{1 - 2 \sigma} \| \Delta_z^2 u \|_{L^2(B_{6 R})}^2 \\
& \quad + \varepsilon^2 \left( [\eta^2 \Delta_z^2 u]_{H^s(B_{6 R})}^2 + |z|^{- 2 s} \| \Deltamz \left( \eta^2 \Delta_z^2 u \right) \|_{L^2(B_{3 R})}^2 \right) \Big].
\end{aligned}
\end{equation}

Now, we turn our attention to~$J$. Arguing as in~\eqref{J5}, we use once again~\eqref{Kbounds},~\eqref{etacutoff} and Young's inequality to obtain
$$
|J| \le c \left[ \delta \left( \| u \|_{L^2(B_{3 R})}^2 + \| u \|_{L^1_{x_0, \beta}(\R^n)}^2 \right) + \delta^{-1} \| \Delta_{-z}^2 \left( \eta^2 \Delta_z^2 u \right) \|_{L^2(B_{3 R})}^2 \right],
$$
for any~$\delta > 0$.
Setting again~$\delta = \varepsilon^{- 2} |z|^{2 s}$, this becomes
\begin{equation} \label{J}
|J| \le c \left[ \varepsilon^{- 2} |z|^{2 s} \left( \| u \|_{L^2(B_{3 R})}^2 + \| u \|_{L^1_{x_0, \beta}(\R^n)}^2 \right) + \varepsilon^2 |z|^{- 2 s} \| \Delta_{-z}^2 \left( \eta^2 \Deltaz u \right) \|_{L^2(B_{3 R})}^2 \right].
\end{equation}

Finally, we use Young's inequality as before to deduce
$$
\left| \int_{B_{3 R}} f(x) \Delta_{-z}^2 \left( \eta^2 U \right)(x) \, dx \right| \le c \left[ \varepsilon^{- 2} |z|^{2 s} \| f \|_{L^2(\Omega)}^2 + \varepsilon^2 |z|^{- 2 s} \| \Delta_{-z}^2 \left( \eta^2 \Delta_z^2 u \right) \|_{L^2(B_{3 R})}^2 \right].
$$

By combining this last estimation,~\eqref{I},~\eqref{J} with~\eqref{plug} and noticing that
$$
\| \Delta_{-z}^2 \left( \eta^2 \Deltaz u \right) \|_{L^2(B_{3 R})} \le 2 \| \Deltamz \left( \eta^2 \Deltaz u \right) \|_{L^2(B_{4 R})},
$$
we find
\begin{equation} \label{final1}
\begin{aligned}
[\eta \Delta_z^2 u]_{H^s(B_{6 R})} & \le c \Big[ \varepsilon \left( [\eta^2 \Delta_z^2 u]_{H^s(B_{6 R})} + |z|^{- s} \| \Deltamz \left( \eta^2 \Delta_z^2 u \right) \|_{L^2(B_{4 R})} \right) \\
& \quad + |z|^{1 / 2 - \sigma} \| \Delta_z^2 u \|_{L^2(B_{6 R})} \\
& \quad + \varepsilon^{- 1} |z|^s \left( \| u \|_{H^s(B_{8 R})} + \| u \|_{L^1_{x_0, \beta}(\R^n)} + \| f \|_{L^2(\Omega)} \right) \Big].
\end{aligned}
\end{equation}
In view of Proposition~\ref{sobnik1}, we have\footnote{Here and in the remainder of the proof we freely swap between some of the equivalent norms of Nikol'skii spaces. In this regard, we recommend the reader to refer to Section~\ref{besovsec} and, in particular, Remark~\ref{nikaltnormrmk}.}
\begin{equation} \label{Delta-z}
\begin{aligned}
\| \Deltamz \left( \eta^2 \Delta_z^2 u \right) \|_{L^2(B_{4 R})} & \le \| \Deltamz \left( \eta^2 \Delta_z^2 u \right) \|_{L^2((B_{5 R})_{-z})} \\
& \le |z|^s [\eta^2 \Delta_z^2 u]_{N^{s, 2}(B_{5 R})} \\
& \le c |z|^s \| \eta^2 \Delta_z^2 u \|_{H^s(B_{5 R})}.
\end{aligned}
\end{equation}
Moreover,
\begin{align*}
& \left| \left( \eta^2 \Delta_z^2 u \right)(x) - \left( \eta^2 \Delta_z^2 u \right)(y) \right|^2 \\
& \hspace{50pt} \le 2 \left( |\eta(x)|^2 \left| \left( \eta \Delta_z^2 u \right)(x) - \left( \eta \Delta_z^2 u \right)(y) \right|^2 + \left| \left( \eta \Delta_z^2 u \right)(y) \right|^2 |\eta(x) - \eta(y)|^2 \right),
\end{align*}
and hence, recalling~\eqref{etacutoff},
\begin{equation} \label{eta2eta}
\begin{aligned}
[\eta^2 \Delta_z^2 u]_{H^s(B_{6 R})}^2 & \le c \left[ [\eta \Delta_z^2 u]_{H^s(B_{6 R})}^2 + \int_{B_{6 R}} |\Delta_z^2 u(y)|^2 \left[ \int_{B_{6 R}} |x - y|^{- n - 2 s + 2} \, dx \right] dy \right] \\
& \le c \left[ [\eta \Delta_z^2 u]_{H^s(B_{6 R})}^2 + \| \Delta_z^2 u \|_{L^2(B_{6 R})}^2 \right].
\end{aligned}
\end{equation}
Consequently, if we choose~$\varepsilon$ suitably small, by~\eqref{Delta-z},~\eqref{eta2eta} and Proposition~\ref{caccioppoliprop}, estimate~\eqref{final1} becomes
\begin{equation} \label{final2}
\begin{aligned}
[\eta \Delta_z^2 u]_{H^s(B_{6 R})} & \le c \bigg[ |z|^{1 / 2 - \sigma} \| \Delta_z^2 u \|_{L^2(B_{6 R})} \\
& \quad + |z|^s \left( \| u \|_{L^2(\Omega)} + \| u \|_{L^1_{x_0, \beta}(\R^n)} + \| f \|_{L^2(\Omega)} \right) \bigg],
\end{aligned}
\end{equation}
where we also employed~\eqref{sigmabound}. Applying again Proposition~\ref{sobnik1},
$$
\| \tensor*[]{\Delta}{_w} \left( \Delta_z^2 u \right) \|_{L^2 ((B_R)_w)} \le |w|^s [\Delta_z^2 u]_{N^{s, 2}(B_R)} \le c |w|^s \| \Delta_z^2 u \|_{H^s(B_R)},
$$
for any~$w \in \R^n$. Taking~$w = z$, from~\eqref{kzinjz},~\eqref{etacutoff},~\eqref{sigmabound} and~\eqref{final2} we then get
\begin{equation} \label{final3}
\begin{aligned}
\| \Delta_z^3 u \|_{L^2((B_R)_{3 z})} & \le \| \Delta_z^3 u \|_{L^2((B_R)_z)} \le c |z|^s \| \Delta_z^2 u \|_{H^s(B_R)} \\
& \le c |z|^s \left( \| \Delta_z^2 u \|_{L^2(B_R)} + [\eta \Delta_z^2 u]_{H^s(B_{6 R})} \right) \\
& \le c \bigg[ |z|^{1 / 2 - \sigma + s} \| \Delta_z^2 u \|_{L^2(B_{6 R})} \\
& \quad + |z|^{2 s} \left( \| u \|_{L^2(\Omega)} + \| u \|_{L^1_{x_0, \beta}(\R^n)} + \| f \|_{L^2(\Omega)} \right) \bigg].
\end{aligned}
\end{equation}

Now we consider separately the two cases~$s \in (0, 1 / 2]$ and~$s \in (1 / 2, 1)$.

In the first situation, we set~$\sigma = 1 / 2$. Notice that the choice is compatible with~\eqref{sigmabound}. By Proposition~\ref{sobnik1},
\begin{equation} \label{Delta2u}
\| \Delta_z^2 u \|_{L^2(B_{6 R})} \le \| \Delta_z^2 u \|_{L^2((B_{7 R})_z)} \le |z|^s [u]_{N^{s, 2}(B_{7 R})} \le c |z|^s \| u \|_{H^s(B_{7 R})}.
\end{equation}
Therefore, from~\eqref{final3}
\begin{equation} \label{finals<}
\| \Delta_z^3 u \|_{L^2((B_R)_{3 z})} \le c |z|^{2 s} \left( [u]_{H^s(B_{56 R})} + \| u \|_{L^2(\Omega)} + \| u \|_{L^1_{x_0, \beta}(\R^n)} + \| f \|_{L^2(\Omega)} \right),
\end{equation}
and thus~$u \in N^{2 s, 2}(B_R)$.

Now we address the more delicate case~$s \in (1 / 2, 1)$. Here we choose~$\sigma = s$ and first deduce from~\eqref{final3} and~\eqref{Delta2u} that
$$
\| \Delta_z^3 u \|_{L^2((B_R)_{3 z})} \le c |z|^{1 / 2 + s} \left( [u]_{H^s(B_{7 R})} + \| u \|_{L^2(\Omega)} + \| u \|_{L^1_{x_0, \beta}(\R^n)} + \| f \|_{L^2(\Omega)} \right).
$$
Note that such a~$\sigma$ is admissible for~\eqref{sigmabound}, since~$s > 1 / 2$. Repeating the same argument with~$B_{8 R}$ in place of~$B_R$, we see that~$u \in N^{1 / 2 + s, 2}(B_{8 R})$ with
$$
[u]_{N^{1 / 2 + s, 2}(B_{8 R})} \le c \left( [u]_{H^s(B_{56 R})} + \| u \|_{L^2(\Omega)} + \| u \|_{L^1_{x_0, \beta}(\R^n)} + \| f \|_{L^2(\Omega)} \right).
$$
Consequently,
\begin{align*}
\| \Delta_z^2 u \|_{L^2(B_{6 R})} & \le \| \Delta_z^2 u \|_{L^2((B_{8 R})_{2 z})} \le |z|^{1 / 2 + s} [u]_{N^{1 / 2 + s}(B_{8 R})} \\
& \le c |z|^{1 / 2 + s} \left( [u]_{H^s(B_{5 6 R})} + \| u \|_{L^2(\Omega)} + \| u \|_{L^1_{x_0, \beta}(\R^n)} + \| f \|_{L^2(\Omega)} \right).
\end{align*}
Using this last estimate in combination with~\eqref{final3} and selecting~$\sigma = 1$ there, again in agreement with~\eqref{sigmabound}, we conclude that~$u \in N^{2 s, 2}(B_R)$ and~\eqref{finals<} is true also for~$s \in (1 / 2, 1)$.

Finally, we use Proposition~\ref{caccioppoliprop} to control the Gagliardo semi-norm on the right hand side of~\eqref{finals<} and recover
\begin{equation} \label{finalsob}
[u]_{N^{2 s, 2}(B_R)} \le c \left( \| u \|_{L^2(\Omega)} + \| u \|_{L^1_{x_0, \beta}(\R^n)} + \| f \|_{L^2(\Omega)} \right).
\end{equation}
Then,~\eqref{nikmainest} follows for a general open~$\Omega' \subset \subset \Omega$ by a standard covering argument.\footnote{Note that the right hand side of~\eqref{finalsob} depends on the norm~$\| \cdot \|_{L^1_{x_0, \beta}(\R^n)}$ which in turn varies with~$x_0$. Consequently, while performing the covering argument one should take care that those norms depend on the centers of the covering balls. However, as noted in Section~\ref{defsec} such norms are all equivalent. The relative compactness of~$\Omega'$ then allows the use of a finite number of balls, thus preventing the blow-up of the constant~$c$.}
\end{proof}

We conclude this section with some brief comments on the technique just displayed.

To achieve the result we tested the equation with a function modelled on the double increment~$\Delta_z^2 u$, which may seem a little unnatural and artificial. In fact, for~$s \in (0, 1/2]$ the first order increment would have been sufficient. On the other hand, when~$s > 1 / 2$ this strategy is no more conclusive, basically since it leads to~$u \in N^{1 / 2 + s, 2}_\loc(\Omega)$ only. In order to take advantage of this intermediate regularity and then gain the extra~\emph{$s - 1 / 2$ derivatives}, we need the order of the increment to be at least~$2$.

\section{Proof of Theorem~\ref{sobmainthm}} \label{sobmainsec}

As previously discussed in Section~\ref{defsec}, Theorem~\ref{sobmainthm} essentially follows from Theorem~\ref{nikmainthm}, in light of the embedding of Proposition~\ref{sobnik2}. The only detail left is that the results of Section~\ref{besovsec} - specifically, Proposition~\ref{sobnik2} - are only proved for sets having smooth boundary.

But this is not a big drawback. As a matter of fact, we know that estimate~\eqref{sobmainest} holds for any domain~$\Omega' \subset \subset \Omega$, with~$\partial \Omega' \in C^\infty$. Then, it can be further extended to any~$\Omega'$, by noticing that it is always possible to find~$\Omega''$ with~$C^\infty$ boundary, such that~$\Omega' \subset \subset \Omega'' \subset \subset \Omega$.

\section{Towards the optimal regularity up to the boundary} \label{optsec}

In this conclusive section we briefly comment on the global Sobolev regularity for the Dirichlet problem driven by~\eqref{maineq}.

For~$x \in \R^n$, we define~$u_s(x) := (x_n)_+^s$. The function~$u_s$ solves
\begin{equation} \label{u_sprob}
\begin{cases}
(-\Delta)^s u_s = 0 & \quad \mbox{in } \R^n_+ := \R^{n - 1} \times (0, +\infty) \\
\,          u_s = 0 & \quad \mbox{in } \R^n \setminus \R^n_+.
\end{cases}
\end{equation}
To see this, we write~$u_s(x) = \mu_s(x_n)$, with~$\mu_s(t) := t_+^s$ as~$t \in \R$, and we compute for~$x \in \R^n_+$
\begin{align*}
(-\Delta)^s u_s(x) & = 2 \, \PV \int_{\R^n} \frac{u_s(x) - u_s(y)}{|x - y|^{n + 2 s}} \, dy \\
& = 2 \, \PV \int_{\R} \frac{\mu_s(x_n) - \mu_s(y_n)}{|x_n - y_n|^{n + 2 s}} \left[ \int_{\R^{n - 1}} \left( 1 + \frac{|x' - y'|^2}{|x_n - y_n|^2} \right)^{- \frac{n + 2 s}{2}} dy' \right] dy_n.
\end{align*}
Note that we use~$x'$ and~$y'$ to indicate the first~$n - 1$ components of~$x$ and~$y$, respectively. Changing variables by setting~$z' := |y_n - x_n|^{- 1} (y' - x')$ in the inner integral, we get
$$
(-\Delta)^s u_s(x) = \varpi_{n, s} (-\Delta)^s \mu_s(x_n),
$$
where
$$
\varpi_{n, s} := \int_{\R^{n - 1}} \left( 1 + |z'|^2 \right)^{- \frac{n + 2 s}{2}} dz',
$$
is a finite constant. Then, the equation in~\eqref{u_sprob} follows from the fact that~$\mu_s$ is~$s$-harmonic in the half-line~$(0, +\infty)$, as showed for instance in~\cite{CRS10, RS14} or~\cite{BV15}.

\medskip

Of course, the function~$u_s$ is of class~$C^{0, s}_\loc(\R^n)$, but not~$C^{0, \alpha}_\loc(\R^n)$, with~$\alpha > s$. On the other hand, the following proposition sheds some light on which could be the optimal Sobolev regularity of~$u_s$, at least when~$s \ge 1 / 2$.

\begin{proposition} \label{optprop}
Let~$s \in [1 / 2, 1)$. Then,~$u_s \notin H^{2 s}_\loc(\overline{\R^n_+})$.
\end{proposition}

\begin{proof}
We focus on the case~$s > 1 / 2$, as when~$s = 1 / 2$ the computation is immediate.

Denoting with~$B'_r(z')$ the~$(n - 1)$-dimensional open ball of radius~$r$ and center~$z'$ - with~$B_r' := B_r'(0)$ as usual - and with~$Q$ the cylinder~$B'_1 \times (0, 1)$, we shall prove that
\begin{equation} \label{optth}
u_s \notin H^{2 s}(Q).
\end{equation}
First, setting
$$
E := \int_0^1 \int_0^1 \frac{|\mu_s'(t) - \mu_s'(r)|^2}{|t - r|^{1 + 2 (2 s - 1)}} \, dt dr,
$$
we claim that
\begin{equation} \label{optclaim}
E \mbox{ is not finite.}
\end{equation}

Assuming for the moment~\eqref{optclaim} to hold, we check that then~\eqref{optth} follows. While for~$n = 1$ this is immediate, the case~$n \ge 2$ requires some comments. Indeed,
\begin{align*}
\| u_s \|_{H^{2 s}(Q)}^2 & \ge \int_Q \int_Q \frac{|\nabla u_s(x) - \nabla u_s(y)|^2}{|x - y|^{n + 2 (2 s - 1)}} \, dx dy = \int_Q \int_Q \frac{|\mu_s'(x_n) - \mu_s'(y_n)|^2}{|x - y|^{n + 2 (2 s - 1)}} \, dx dy \\
& = \int_0^1 \int_0^1 |\mu_s'(x_n) - \mu_s'(y_n)|^2 \left( \int_{B_1'} \int_{B_1'} \frac{dx' dy'}{\left( |x_n - y_n|^2 + |x' - y'|^2 \right)^{\frac{n}{2} + 2 s - 1}} \right) dx_n dy_n.
\end{align*}
For~$\delta \in (0, 1/2)$ we consider the set
$$
S(\delta) := \left\{ (x', y') \in B_1' \times B_1' : |x' - y'| < \delta \right\} \subset \R^{n - 1} \times \R^{n - 1},
$$
and we estimate its measure by computing
\begin{align*}
|S(\delta)| & = \int_{B_1'} \left( \int_{B_1' \cap B_\delta'(x')} dy' \right) dx' \ge \int_{B_{1 - \delta}'} \left( \int_{B_\delta'(x')} dy' \right) dx' \\
& = |B_1'|^2 (1 - \delta)^{n - 1} \delta^{n - 1} \ge 2^{1 - n} |B_1'|^2 \delta^{n - 1}.
\end{align*}
Noticing that on~$S(|x_n - y_n| / 4)$ it holds
$$
|x_n - y_n|^2 + |x' - y'|^2 \le \frac{17}{16} \, |x_n - y_n|^2,
$$
and that~$|x_n - y_n| / 4 \le 1 / 2$, we finally obtain
\begin{align*}
\| u_s \|_{H^{2 s}(Q)}^2 & \ge \left( \frac{16}{17} \right)^{\frac{n + 2 s}{2}} \int_0^1 \int_0^1 \frac{|\mu_s'(x_n) - \mu_s'(y_n)|^2}{|x_n - y_n|^{n + 2 (2 s - 1)}} \left| S \left( \frac{|x_n - y_n|}{4} \right) \right| dx_n dy_n \\
& \ge \left( \frac{16}{17} \right)^{\frac{n + 2 s}{2}} 8^{1 - n} |B_1'|^2 \, E.
\end{align*}
Thus,~\eqref{optth} is valid.

To complete the proof of the proposition, we are only left to show that~\eqref{optclaim} is true. To do this, we first note that, for~$t > 0,$
\begin{align*}
\mu_s'(t) & = s t^{s - 1}, \\
\mu_s''(t) & = s(s - 1) t^{s - 2} < 0.
\end{align*}
Accordingly,~$\mu_s'$ is decreasing and for~$0 < r < t < 1$ we have
\begin{align*}
|\mu_s'(t) - \mu_s'(r)| & = \mu_s'(r) - \mu_s'(t) = - \int_r^t \mu_s''(\tau) \, d\tau \\
& = s (1 - s) \int_r^t \tau^{s - 2} \, d\tau \ge s (1 - s) t^{s - 2} (t - r),
\end{align*}
so that
\begin{align*}
E & \ge s^2 (1 - s)^2 \int_0^1 t^{2 (s - 2)} \left( \int_0^t (t - r)^{3 - 4 s} dr \right) dt = \frac{s^2 (1 - s)}{4} \int_0^1 t^{- 2 s} dt.
\end{align*}
Claim~\eqref{optclaim} then follows, since the integral on the right hand side of the above inequality does not converge.
\end{proof}

We remark that, for~$s \in (0, 1 / 2)$, an almost identical argumentation leads to the conclusion that~$u_s \notin H^{s + 1 / 2}_\loc(\overline{\R^n_+})$.

\end{document}